\documentclass{amsart}
\usepackage{amssymb,amsbsy}
\usepackage[matrix,arrow,curve]{xy}
\usepackage{verbatim}
\usepackage{a4wide}
\usepackage[applemac]{inputenc}
\usepackage{enumitem}

\newtheorem{theorem}{Theorem}[section]
\newtheorem{lemma}[theorem]{Lemma}

\newtheorem{corollary}[theorem]{Corollary}

\newtheorem{proposition}[theorem]{Proposition}

\def\eps{\varepsilon}

\def\GL{\mathrm{GL}}

\def\C{\mathbf{C}}
\def\Z{\mathbf{Z}}
\def\N{\mathbf{N}}
\def\kk{\mathbf{k}}

\def\Q{\mathbf{Q}}
\def\C{\mathbf{C}}

\def\Ker{\mathrm{Ker}\,}

\def\into{\hookrightarrow}
\def\onto{\twoheadrightarrow}

\def\ii{\mathrm{i}}
\def\la{\lambda}

\newcommand\bpi{{\boldsymbol{\pi}}}

\date{February 4, 2017.}

\begin{document}
\centerline{}

\title{Lattice extensions of Hecke algebras}
\author[I.~Marin]{Ivan Marin}
\address{LAMFA, Universit\'e de Picardie-Jules Verne, Amiens, France}
\email{ivan.marin@u-picardie.fr}
\medskip

\begin{abstract} 
We investigate the extensions of the Hecke algebras of finite (complex) reflection groups
by lattices of reflection subgroups that we introduced, for some of them, in our previous work on the Yokonuma-Hecke algebras and their connections with Artin groups.
When the Hecke algebra is attached to the symmetric group, and the lattice contains all reflection subgroups, then these algebras are the diagram algebras of braids and ties of Aicardi and Juyumaya. We prove a stucture theorem
for these algebras, generalizing a result of Espinoza and Ryom-Hansen from the case of the symmetric group to the general case. We prove that these algebras are symmetric algebras at least when $W$ is a Coxeter group, and in general under the trace conjecture of Brou\'e, Malle and Michel.
\end{abstract}

\maketitle

\tableofcontents

\section{Introduction}

Let $W$ be finite complex reflection group, for instance a finite Coxeter group.
Let $B$ denote the braid group associated to $W$ in the sense of Broué-Malle-Rouquier
(see \cite{BMR}), which in the case of a finite Coxeter group coincides with the Artin group
attached to it. We denote $\pi : B \to W$ the natural projection.

The object of this paper is to introduce and analyse a family of algebras denoted
$\mathcal{C}(W,\mathcal{L})$, where $\mathcal{L}$ is a finite join semi-lattice 
which lies inside the poset made of the full reflection subgroups of $W$, ordered by inclusion. 
Here a reflection subgroup of $W$
is called \emph{full} if, for any reflection in this subgroup, all the (pseudo-)reflections
with the same reflecting hyperplane belong to it. The semi-lattice
$\mathcal{L}$ is additionnally supposed to be stable under the natural
action of $W$ on the lattice of reflection subgroups, and to contain all the cyclic (full) reflection subgroups, and the trivial subgroup as well. Such a semi-lattice will be called
an \emph{admissible} semi-lattice.

Let $\mathcal{A}$ denote the hyperplane arrangement attached to 
$W$, namely the collection of its reflecting hyperplanes.
Let $\kk$ be a commutative ring with $1$,
containing elements $a_{H,i}$ where $H \in \mathcal{A}$, $0 \leq i < m_H$
where $m_H$ is the order of the cyclic subgroup of $W$ fixing $H$, with
the convention that $a_{H,i} = a_{w(H),i}$ for every $H \in \mathcal{A}, w \in W$
and $a_{H,0}$ is invertible inside $\kk$. Let $R$ denote the generic ring of Laurent
polynomials with integer coefficients
$\Z[a_{H,i}, a_{H,0}^{\pm 1}]$, with the same conventions. Our conditions on $\kk$ mean
that it is a $R$-algebra. We now define $\kk$-algebras
$\mathcal{C}_{\kk}(W, \mathcal{L})$, with the convention that
$\mathcal{C}(W, \mathcal{L}) = \mathcal{C}_R(W,\mathcal{L})$.

These algebras are defined as follows. First consider the algebra $\kk \mathcal{L}$
defined as the free $\kk$-module with basis elements $e_{\la}, \la \in \mathcal{L}$,
and where the multiplication is defined by $e_{\la} e_{\mu} = e_{\la \vee \mu}$. 
This is
sometimes called the Möbius algebra of $\mathcal{L}$. 
Elements of $\mathcal{L}$
can be identified with the collection of reflecting hyperplanes attached to them,
and we let $e_H = e_{\{ H \} }$ denote the idempotent attached
to the subgroup fixing $H \in \mathcal{A}$. We shall use this identification whenever
it is convenient to us.

By definition $W$ acts
by automorphisms on $\kk \mathcal{L}$, hence so does $B$, and one
can form the semidirect product $\kk B \ltimes \kk \mathcal{L}$. The algebras
$\mathcal{C}_{\kk}(W,\mathcal{L})$ are defined as the quotient of 
$\kk B \ltimes \kk \mathcal{L}$ by the two-sided ideal $\mathfrak{J}$ generated by the elements $\sigma^{m_H} - 1 - Q_s(\sigma)e_{H }$ where $\sigma$ runs among the braided reflections of $B$, $s = \pi(\sigma)$ is the corresponding pseudo-reflection,
$H = \Ker(s-1)$, and  $Q_s(X) = \sum_{k=0}^{m_H-1} a_{H,k} X^k - 1\in \kk[X]$
(see section \ref{sect:imageideal} for more details).

Let $\mathfrak{J}_H$ is the ideal of $\kk B$
generated by the $\sigma^{m_H} - 1 - Q_s(\sigma) = \sigma^{m_H} - \sum_{k=0}^{m_H-1} a_{H,k} \sigma^k$. This quotient $H_k(W) = (\kk B)/\mathfrak{J}_H$ 
is by definition the Hecke algebra attached to $W$ in the sense of Broué-Malle-Rouquier,
and is the usual Iwahori-Hecke algebra of $W$ when $W$ is a finite Coxeter group. The following preliminary result explains the title, making our algebras appear as natural
extensions of the Hecke algebra $H_{\kk}(W)$.

\begin{proposition} Let $\mathcal{L}$ be an admissible lattice for $W$.
There exists a surjective algebra morphism $\mathcal{C}_{\kk}(W, \mathcal{L}) \to
H_{\kk}(W)$, defining a \emph{split} extension of $H_{\kk}(W)$.
\end{proposition}
\begin{proof} The natural augmentation map $\eta : \kk \mathcal{L} \to \kk$ defined
by $e_{\la} \mapsto 1$ induces surjective morphisms of $\kk$-algebras
$\eta : \kk B \ltimes \kk \mathcal{L} \to \kk B$ and
$\mathcal{C}_{\kk}(W,\mathcal{L}) \to (\kk B)/\mathfrak{J}_H = H_{\kk}(W)$.
The splitting comes from the fact that the assumptions on $\mathcal{L}$
imply that $W$ belongs to $\mathcal{L}$, as the join of all the full cyclic subgroups.
Then, the non-unital algebra morphism $\kk B \to \mathcal{C}_{\kk}(W, \mathcal{L})$
defined by $b \mapsto b e_W$ is easily checked to factorize through $H_{\kk}(W)$
and to provide a splitting.
\end{proof}

Our first result is a structure theorem of the following form, where the $\kk \tilde{H}_{x_*}$ are slight generalizations
of the Hecke algebras attached to elements of $\mathcal{L}$ and $x_* \in \mathcal{L}$
is a representative of the orbit $X \in \mathcal{L}/W$.

\begin{theorem} There exists an isomorphism of $\kk$-algebras
$$
\mathcal{C}_{\kk}(W,\mathcal{L}) \simeq 
\bigoplus_{X \in \mathcal{L}/W} Mat_X(\kk \tilde{H}_{x_*}).
$$
\end{theorem}

When $\mathcal{L}$ is the lattice of the reflection subgroups of a finite Coxeter
group, the algebras $\mathcal{C}(W,\mathcal{L})$ were introduced in \cite{YH1}, under
the name $\mathcal{C}_W$ and
using a presentation by generators and relations, and proven to be generically
semisimple. 
When $W$ is the symmetric group, $\mathcal{C}_W$
coincides with the diagram algebra of braids and ties of Aicardi and Juyumaya
(see \cite{AICARDIJUYU,AICARDIJUYU2}).
Therefore the above theorem is a generalization of a theorem of Espinoza
and Ryom-Hansen (see \cite{ERH}), and was actually motivated by it. Note that, when
$W$ is the symmetric group, the lattice of parabolic subgroups coincides with
the lattice of reflection subgroups.

We now return to the general case.  
We let $K$ denote the field
of fractions of $R$ and $\bar{K}$ an algebraic closure of $K$. 
The \emph{BMR freeness conjecture} states that $H_{\kk}(W)$ is a free $\kk$-module
of rank $|W|$, and implies that $H_{\kk}(W)$ is generically semisimple. Up to extending
the ring of definition $R$ to a slightly larger Laurent polynomial ring $R_u$,
an additional conjecture of Broué-Malle-Michel, which we recall in detail in section \ref{sect:traces}, states that $H_{\kk}(W)$ is a symmetric
algebra when $\kk$ is a $R_u$-algebra, with a trace enjoying some uniqueness conditions. Of course both conjectures
are true when $W$ is a finite Coxeter group. 

When $\mathcal{L}$ is the lattice of parabolic subgroups of a finite complex reflection groups, the algebra $\mathcal{C}(W,\mathcal{L})$
was introduced and called $\mathcal{C}_W^p$ in \cite{YH1}. It was conjectured
there that $\mathcal{C}_W^p$ is a free $R$-module of rank $|W| \times |\mathcal{L}_p|$,
where $\mathcal{L}_p$ denotes the lattice of parabolic subgroups. A consequence
of the above theorem is then the following one.
We denote $W_{x_*}<W$ the stabilizer of $x_*$.

\begin{theorem} \label{theo:structtheoremBMRintro}
The algebra $\mathcal{C}_{\kk}(W,\mathcal{L})$
is a free $\kk$-module of
finite rank if and only if the BMR freeness conjecture holds over $\kk$ for every $x \in \mathcal{L}$
 (this is in particular the case when $W$
is a Coxeter group). In that case, its rank is $|W| \times |\mathcal{L}|$, and $\mathcal{C}_{\kk}(W,\mathcal{L})$ is semisimple when $\kk$ is an extension of $K$,
and 
$$\mathcal{C}_{\bar{K}}(W,\mathcal{L}) \simeq \bar{K} W \ltimes \bar{K} \mathcal{L}
\simeq \bigoplus_{X \in \mathcal{L}/W} Mat_X(\bar{K} W_{x_*}).
$$
\end{theorem}

The BMR freeness conjecture is now proved for all irreducible reflection groups
but the ones of Shephard-Todd types $G_{17}$, $G_{18}$ and $G_{19}$ (see \cite{ARIKI,
ARIKIKOIKE,CYCLO,HECKECUBIQUE,MARINPFEIFFER,CHAVLITHESE,CHAVLICRAS,CB3}),
therefore the above statement is actually almost unconditional, and reduces
the proof of conjecture 5.10 in \cite{YH1} to the original BMR freeness conjecture.

We finally (conditionnally) prove that these algebras are symmetric algebras. We call 
\emph{strong
freeness conjecture} for $W$ the statement that $H_R(W)$ admits a basis
originating from elements of $B$. It turns out that the status of this conjecture
is exactly the same as the original BMR freeness conjecture : for every group for
which the BMR freeness conjecture has been proved so far, the proof provides
a convenient basis.

\begin{theorem} Assume that the strong freeness conjecture as well as the Brou\'e-Malle-Michel trace conjecture holds for all $x \in \mathcal{L}$. This is in particular the case
if $W$ is a finite Coxeter group. Then, for any commutative $R_u$-algebra $\kk$, the algebra $\mathcal{C}_{\kk}(W,\mathcal{L})$
is a symmetric algebra.
\end{theorem}

As an immediate corollary, we get that the diagram algebra of `braids and ties' is
a symmetric algebra as well.

\medskip

{\bf Acknowledgements.} I thank J.-Y. H\'ee and S. Bouc for discussions about
root systems and lattices. I thank M. Calvez, A. Navas and J. Juyumaya for
their invitation at the SUMA'16 conference in Valparaiso, where the original idea
of this paper emerged.

\section{Structure}

\subsection{Semidirect extensions of group algebras by abelian algebras}
\label{sect:prelimalgebras}

In this section, we first expose fairly general results, which are basically folklore,
and which are needed in the sequel. To start with, the following proposition is an explicit version of what is known
in the realm of the representation theory of finite groups as Mackey-Wigner's
method of ``little groups'' (see \cite{SERRE} \S 8.2). It can be seen as an
explicit Morita equivalence (see \cite{CABANESENGUEHARD} ex. 18.6). It is stated
and proved in detail in \cite{CABANESMARIN}, proposition 3.4, in the case $G$ is finite.
We explain below the additional arguments which are needed in the general case.

\begin{proposition} \label{prop:wigner}
Let $G$ be a group acting transitively (on the left) on a finite set $X$. Let $\kk$ be a commutative ring with $1$, and let $A$ be the 
$\kk$-algebra $G\ltimes \kk^X$ where $\kk^X=\oplus_{x\in X}\kk\epsilon_x$ is endowed with the product law
($\epsilon_x\epsilon_{x'}=\delta_{x,x'}\epsilon_x$) and the action of $G$ is induced by the one on $X$. Then any choice 
of $x_*\in X$ with stabilizer $G_0\subseteq G$ and any choice of a ``section" $\tau \colon X\to G$ such that $\tau(x).x_*=x$ for all $x\in X$, 
define a unique isomorphism $$ \theta : A\longrightarrow {\rm Mat}_X(\kk G_0)$$ sending each $\epsilon_{x}\in \kk^X$ ($x\in X$) to 
$\theta (\epsilon_x):=E_{x,x}$, and each $g\in G$ to $$\theta (g):=\sum_{x\in X}\tau(gx)^{-1}g.\tau(x) E_{gx,x}$$ 
(where $E_{x,y}\in{\rm Mat}_X(\kk )$ is the elementary matrix corresponding to $x,y\in X$).
\end{proposition}
\begin{proof}
The proof given in proposition 3.4 of \cite{CABANESMARIN} that
$\theta$ is a surjective morphism does not use any finiteness assumption on $G$. It
therefore remains to prove that $\theta$ is injective. We prove this directly
as follows. A $\kk$-basis of $A$ is given by the $g \eps_{\alpha}$ for $g \in G$
and $\alpha \in X$, and by definition
$$
\theta(g \eps_{\alpha}) = \sum_{x \in X} \tau(gx)^{-1}g\tau(x) E_{g.x,x} E_{\alpha,\alpha}
= \tau(g\alpha)^{-1}g\tau(\alpha) E_{g\alpha,\alpha}.
$$
It follows that a general linear combination $\sum_{g, \alpha} \la_{g,\alpha} g\eps_{\alpha}$
belongs to $\Ker \theta$ iff 
$$
0 = \sum_{g,\alpha} \la_{g,\alpha} \tau(g\alpha)^{-1}g\tau(\alpha)E_{g\alpha,\alpha}
$$
which means that, for all $\alpha \in X$,
$$
\sum_{g \in G} \la_{g,\alpha} \tau(g \alpha)^{-1}g\tau(\alpha)E_{g\alpha,\alpha} = 0.
$$
Let us fix such an $\alpha \in X$. For every $\beta \in X$ we have
$$
0 = \sum_{g \ | \ g.\alpha = \beta } \la_{g,\alpha} \tau(g \alpha)^{-1}g\tau(\alpha)E_{g\alpha,\alpha} 
$$
namely
$$
0 = \tau(\beta)^{-1}\left(\sum_{g \ | \ g.\alpha = \beta } \la_{g,\alpha} g\right) \tau(\alpha)
$$
which implies that, for all $g \in G$, we have $\la_{g,\alpha} = 0$. Since this
holds for every $\alpha \in X$ we get the conclusion.
\end{proof}

Let $L$ be a join semilattice. That is, we have a finite partially ordered set $L$
for which there exists a least upper bound $x \vee y$ for every two $x,y \in L$.
Let $M$ be the semigroup with elements $e_{\la}, \la \in L$ and product law
$e_{\la} e_{\mu} = e_{\la \vee \mu}$. 
Such a semigroup is sometimes called a band.

If $L$ is acted upon by some group $G$ in an order-preserving way (that is $x \leq y \Rightarrow g.x \leq g.y$ for all $x,y \in L$ and $g \in G$)
then $M$ is acted upon by $G$, so that we can form the algebra $ \kk M \rtimes \kk G$. Up
to exchanging meet and join, the algebra $\kk M$ is the Möbius algebra as in \cite{STANLEY}, definition 3.9.1. We recall from \cite{YH1} a $G$-equivariant version
of the classical isomorphism $\kk M \simeq \kk^L$ of e.g. \cite{STANLEY}, theorem 3.9.2.
Here $\kk^{L}$ is the algebra of $\kk$-valued functions on $L$, that is the direct product of a collection indexed by the elements of $L$ of copies of the $\kk$-algebra $\kk$. As before,
to $\la \in L$ we associate $\eps_{\la} \in \kk^{L}$ defined by $\eps_{\la}(\la') = \delta_{\la,\la'}$
if $\la' \in L$.

\begin{proposition}  (see \cite{YH1}, proposition 3.9) \label{prop:isomobius} Let $M$ be the band associated to a 
finite join semilattice $L$. For every commutative ring $\kk$, the
semigroup algebra $\kk M$ is isomorphic to $\kk^{L}$.
If $L$ is acted upon
by some group $G$ as above, then $\kk M \rtimes \kk G \simeq \kk^{L} \rtimes \kk G$,
the isomorphism being given by $g \mapsto g$ for $g \in G$ and
$e_{\la} \mapsto \sum_{\la \leq \mu} \eps_{\mu}$.
\end{proposition}

By decomposing $L$ as a disjoint union of $G$-orbits, by combining these
two results one gets that $\kk M \rtimes \kk G$ is isomorphic to
a direct sum of $|L/G|$ matrix algebras. This will turn out to be the main result
from general algebra that is needed to prove our structure theorem.

\subsection{Braid groups of reflection subgroups}
\label{sect:braidsubgroups}

\bigskip

Let $W_0$ be a reflection subgroup of the reflection group $W$,
and $G$ a subgroup with $W_0 < G < W$ normalizing $W_0$. 
For convenience we endow $\C^n$ with a $W$-invariant unitary form.

The hyperplane complement associated to $W$ is denoted $X = \C^n \setminus \bigcup
\mathcal{A}$, and we let $x_0 \in X$ denote the chosen base-point,
so that $B = \pi_1(X/W,x_0)$.
Let $L \subset \C^n$ denote the fixed points set of $W_0$,
namely the intersection of the set $\mathcal{A}_L$ of all the reflecting hyperplanes associated to the reflections in $W_0$. 
Since $G$ normalizes $W_0$ we have $g(L) = L$ for all $g \in G$.
We let $X_0 \subset L^{\perp}$ denote the hyperplane complement associated to $W_0$ viewed as a reflection subgroup acting on $L^{\perp}$.  We have $X_0 = L^{\perp} \setminus \bigcup \mathcal{A}_L$.

Let $X^0 = \C^n \setminus \bigcup \mathcal{A}_L$, and $x_{00}$ the orthogonal
projection of $x_0$ on $L^{\perp}$. We write $x_0 = x_1 + x_{00}$, with $x_1 \in L$. Since $x_0 \not\in \bigcup \mathcal{A}_L$
we have $x_{00} \in X_0$, and the braid groups of $W_0 < \GL(L^{\perp})$
can be defined as $P_0 = \pi_1(X_0,x_{00})$, $B_0 = \pi_1(X_0/W_0,x_{00})$.
The inclusion map $(X_0,x_{00}) \subset (X^0,x_{00}+L)$ is a $W_0$-equivariant deformation
retract through $(z,t) \mapsto z_{L^{\perp}}+tz_L$ where $z_{L^{\perp}}$ and $z_L$
denote the orthogonal projections of $z$ on $L^{\perp}$ and $L$, respectively.
Since $x_{00}+L$ is retractable to $x_0$, it follows that
this inclusion provides an isomorphism $P_0 \simeq \pi_1(X^0,x_0)$ and, because
of $W_0$-equivariance, an isomorphism $B_0 \simeq \pi_1(X^0/W_0,x_0)$.

Since $W_0$ is normal inside $G$, the projection map $X/W_0 \to X/G$ is a Galois
covering, and we get a short exact sequence $1 \to \pi_1(X/W_0,x_0) \to \pi_1(X/G,x_0) \to
G/W_0 \to 1$.

We consider the $G$-equivariant inclusion $(X,x_0) \subset (X^0,x_0)$. By standard
arguments (see e.g. \cite{COGOLLUDO} proposition 2.2, or \cite{BESSIS}) we know that
the induced map $P = \pi_1(X,x_0) \to \pi_1(X^0,x_0)$ is surjective,
and that its kernel $K$ is normally generated by the meridians around
the hyperplanes in $\mathcal{A}_L^c$. Since the following diagram is commutative
$$\xymatrix{
 & & 1 \ar[d] & 1 \ar[d] \\
 1 \ar[r] & K \ar@{=}[d] \ar[r] & \pi_1(X,x_0) \ar[d] \ar[r] & \pi_1(X^0,x_0) \ar[d] \ar[r] & 1 \\
 1 \ar[r] & K \ar[r] & \pi_1(X/W_0,x_0) \ar[d] \ar[r] & \pi_1(X^0/W_0,x_0) \ar[d] \ar[r] & 1 \\
  & & W_0 \ar[d] \ar@{=}[r] & W_0 \ar[d] \\
  &  & 1 & 1 
}
$$
with the two columns and the top row being short exact sequences, it follows that the second
row is exact and $K = \Ker(\pi_1(X/W_0,x_0) \to \pi_1(X^0/W_0,x_0))$.
Inside $\pi_1(X/W_0,x_0)$, the collection of meridians generating $K$
become the collection of the elements $\sigma^{m_{\sigma}}$ where
$\sigma$ runs among the collection of (distinguished) braided reflections around
the hyperplanes in $\mathcal{A}_L^c$ and $m_{\sigma}$ is the order of their image in 
$W_0 \subset W$.

Since $G$ stabilizes $\mathcal{A}_L$, the image of $K$ under the injective map $\pi_1(X/W_0,x_0) \to \pi_1(X/G,x_0)$ is a normal subgroup of $\pi_1(X/G,x_0)$, that
we still denote $K$. We define the \emph{generalized braid group associated to $G$} and denote $B_G$ the quotient group $\pi_1(X/G,x_0)/K$.

Let us consider the projection map $\pi : B \to W$.
By the above description, $B_G$ is the quotient
of $\hat{B}_G = \pi_1(X/G,x_0) = \pi^{-1}(G)$ by $K$, and the short exact
sequence $1 \to \pi_1(X/W_0) \to \pi_1(X/G,x_0)/K \to G/W_0 \to 1$
induces a short exact sequence $1 \to \pi_1(X/W_0)/K \to B_G \to G/W_0 \to 1$.
Identifying $\pi_1(X/W_0)/K$ with $\pi_1(X^0/W_0,K)$ we get
a short exact sequence $1 \to B_0 \to B_G \to G/W_0 \to 1$.

\bigskip

We now consider the central element $\bpi_0 \in P_0$ defined
as the class inside $P_0 = \pi_1(X_0,x_{00})$ of the loop $\gamma_0(t) = x_{00}\exp(2 \ii \pi t)$. By the above identifications, it is identified inside $\pi_1(X^0,x_0)$
with the path $\gamma_1 \star \gamma_0 \star \gamma_1^{-1}$,
where $\gamma_1(t) = x_{00}+t x_1$ (recall that $x_0 = x_1 + x_{00}$).
We prove that it remains a central element inside 
$B_G= \pi_1(X/G,x_0)/K$.

For this, let us consider a path $\gamma : x_0 \leadsto g.x_0$ inside $X$. We need
to prove that the composite $ \gamma^{-1} \star (g.\gamma_1 \star g.\gamma_0 \star g.\gamma_1^{-1})^{-1} \star \gamma \star (\gamma_1 \star \gamma_0 \star \gamma_1^{-1})
$,
which is a path $x_0 \leadsto x_0$ inside $X$, belongs to $K$.
This means that its class must be $0$ inside $\pi_1(X,x_0)/K = \pi_1(X^0,x_0)$. Therefore
we need to prove that 
$\gamma \star \gamma_1 \star \gamma_0 \star \gamma_1^{-1} : x_0 \leadsto g.x_0$ is homotopic to 
$g.\gamma_1 \star g.\gamma_0 \star g.\gamma_1^{-1} \star \gamma$ inside $X^0$.
For this, consider the following map $\tilde{H} : [0,3]\times [0,1] \to X^0$
defined, for $t,u \in [0,1]$, by $\tilde{H}(t,u) = \gamma(u)_{L^{\perp}}+(1-t)\gamma(u)_L$,
$\tilde{H}(1+t,u) = \gamma(u)_{L^{\perp}} \exp(2 \ii \pi t)$, $\tilde{H}(2+t,u) = 
\gamma(u)_{L^{\perp}}+t\gamma(u)_L$. It is not difficult to check
that indeed $\tilde{H}(t,u) \in X^0$ for all $t,u$, and that $\tilde{H}$ is continuous.
Moreover, the boundary of the rectangle $[0,3]\times[0,1]$ has for image the union of the two
paths we are interested in. It follows that these two paths are homotopic, which
proves our claim.

\subsection{Proof of the structure theorem}

\subsubsection{Generalized Hecke algebras}

We now attach to an admissible lattice $\mathcal{L}$ the following datas. To each $x \in
\mathcal{L}$ we attach
\begin{itemize} 
\item the ring $R_x= \Z[a_{H,i}, a_{H,0}^{\pm 1}]$ where $H$ runs
among all $H \in x$, and $1 \leq i \leq m_H-1$
\item the stabilizer $G_x < W$ of $x \in \mathcal{L}$ and the group $\hat{B}_x = \hat{B}_{G_x} = \pi^{-1}(G_x)$ associated to 
$W_0 < G_x < N_W(W_0)$, where $W_0$ is the full reflection subgroup associated to $x \in \mathcal{L}$.
\item the group $B_x = B_{G_x}$ as in the previous section.
\end{itemize}

\bigskip

The generalized Hecke algebra $\tilde{H}_x$ associated to $x \in \mathcal{L}$ is then defined as 
the quotient of the group algebra  $R_{x} B_x$ by the ideal generated by the Hecke relations 
$\sigma^{m_H} - \sum a_{H,i} \sigma^i$ for $\sigma$ a braided reflection with respect
to an hyperplane in $x$. 
Equivalently,
it is the quotient of the group algebra $R_x \hat{B}_x$ by the relations
$\sigma^{m_H} - \sum a_{H,i} \sigma^i$ for $\sigma$ a braided reflection with respect
to an hyperplane of $x$ and $\sigma^{m_H}=1$
for $\sigma$ a braided reflection with respect
to an hyperplane of $\mathcal{A} \setminus x$.

Now recall the short exact sequence $1 \to B_0 \to B_x \to G/W_0 \to 1$, 
and consider the induced
injective map $R_x B_0 \to R_x B_x$. We let $\mathfrak{h}_0$ denote the ideal of $R_x B_0$
generated by the $\sigma^{m_H} - \sum a_{H,i} \sigma^i$ for $\sigma$ a braided reflection with respect
to an hyperplane of $x$.
By definition the quotient algebra $H_0 = R_x B_0/\mathfrak{h}_0$ 
the usual Hecke algebra associated to $W_0$.
We let $\mathfrak{h}_x$ the ideal of $R_x B_x$ generated by the same elements,
and choose a system $b_1,\dots,b_m$ of representatives inside $B_x$
of $B_x/B_0 \simeq G/W_0$. Since the generating set of $\mathfrak{h}_x$
is stable under $B_x$-conjugation, we have $\mathfrak{h}_x = \bigoplus_{i=1}^m b_i \mathfrak{h}_0$. This implies that, as a right $R_x B_0$-module,
$\tilde{H}_x = \bigoplus_{i=1}^m (b_i (R_x B_0))/(b_i \mathfrak{h}_0)
= \bigoplus_{i=1}^m (b_i (R_x B_0))/(b_i \mathfrak{h}_0)$.
Now, $(b_i (R_x B_0))/(b_i \mathfrak{h}_0)$ contains (the class of) $b_i$ and
is clearly a free $H_0$-module of rank $1$. This proves that $\tilde{H}_x = R_x B_x/\mathfrak{h}_x$
is a free $H_0$-module of rank $|G/W_0|$. In particular, $\tilde{H}_x$
is a free $R_x$-module of rank $|G|$ if and only if $H_0$ is a free $R_x$-module
of rank $|W_0|$. This latter assumption is exactly the BMR freeness conjecture for $W_0$.

\subsubsection{Image of the defining ideal}
\label{sect:imageideal}

Let $\mathcal{L}$ be an admissible lattice. The group $B$ acts
on $\mathcal{L}$ via the natural projection map $B \to W$.
We denote $\mathfrak{J}$ the ideal of $\kk B \ltimes \kk \mathcal{L}$ 
generated by the elements $\sigma^m - 1 - Q_s(\sigma)e_H$
where
\begin{itemize}
\item $s$ runs among the distinguished pseudo-reflections of $W$,
\item $\sigma$ is a braided reflection attached to it, 
\item $H = \Ker(s-1)$ is the
fixed hyperplane, and 
\item $e_H \in \kk \mathcal{L}$ is the idempotent
attached to $\{ H \}  \in  \mathcal{L}$
\item $m$ is the order of $s$
\item $Q_s(X) = \sum_{k=0}^{m-1} a_{H,k} X^k - 1$, where $\prod_{k=1}^{m} (X-u_{s,i}) = X^m +\sum_{k=0}^{m-1} a_{H,k} X^k$.
\end{itemize}

Let $s$ be a reflection, and $m$ its order. Let $1 \leq k < m$. For any hyperplane $H \in \mathcal{A}$, we have $s(H) = H \Leftrightarrow s^k(H) = H$. It follows that,
for every $x \in \mathcal{L}$, we have $s.x = x \Leftrightarrow s^k.x = x$.
We consider the composite $\theta$ of the maps provided by propositions \ref{prop:isomobius} and \ref{prop:wigner}
$$
\kk B \ltimes \kk \mathcal{L} \to \kk B\ltimes  \kk^{\mathcal{L}} \to \kk B \ltimes \kk^X \to Mat_X(\kk \hat{B}_{x_*})
$$
where $\hat{B}_{x_*} =\hat{B}_{G_{x_*}} = \pi^{-1}(G_{x_*})$ is the stabilizer of $x_* \in \mathcal{L}$
and $X$ is the orbit of $x_*$ under $B$ (or $W$). We have, for all $r \in \Z$, 
$$
\theta(e_H) = \sum_{\stackrel{x \in X}{H \in x}} E_{x,x}
\ \ \ \ \mbox{and} \ \ \ \ 
\theta(\sigma^r) = \sum_{x \in X} \tau(s^r.x)^{-1} \sigma^r \tau(x) E_{s^r.x,x}.
$$
Since $H \in x \Rightarrow s^r .x = x$,
this implies
$$
\theta(\sigma^r e_H) = \sum_{\stackrel{x \in X}{H \in x}}\tau(x)^{-1} \sigma^r \tau(x) E_{x,x}
\ \ \ \ \mbox{    and    }\ \ \ \ 
\theta(Q_s(\sigma) e_H) = \sum_{\stackrel{x \in X}{H \in x}}\tau(x)^{-1} Q_s(\sigma) \tau(x) E_{x,x}
$$
hence the image under $\theta$ of $\sigma^m - 1 - Q_s(\sigma)e_H$ is equal to
$$
\sum_{\stackrel{x \in X}{H \not\in x}} \tau(x)^{-1}(\sigma^m-1)\tau(x) E_{x,x}
+ \sum_{\stackrel{x \in X}{H \in x}} \tau(x)^{-1}(\sigma^m-1-Q_s(\sigma))\tau(x)E_{x,x}.
$$

Now recall the elementary fact that, for any ring $A$ with $1$ (commutative or not),
the twosided ideal of the matrix algebra $Mat_N(A)$ generated by a collection $S^{\alpha}, \alpha \in F$
of matrices $S^{\alpha} = (S^{\alpha}_{i,j})_{1 \leq i, j \leq N}$ is equal to
$Mat_N(I)$ where $I$ is the twosided ideal of $A$ generated by the
$S^{\alpha}_{i,j}$ for $\alpha \in F$, $1 \leq i,j \leq N$.
If follows that image of the ideal $\mathfrak{J}$ inside $Mat_X(\kk \hat{B}_{x_*})$
is $Mat_X(\mathfrak{J}_X)$
where $\mathfrak{J}_X$ is the ideal of $\kk \hat{B}_{x_*}$ generated by the $\sigma_x^m-1$
for $H \not\in x, x \in X$ and the $\sigma_x^m-1-Q_s(\sigma_x)$ for $H \in x, x \in X$,
where $\sigma_x = \tau(x)^{-1}\sigma\tau(x)$. This is the
same as the ideal of $\kk \hat{B}_{x_*}$ generated by the $\sigma^m-1$
for $\sigma$ a braided reflection around some $H \not\in x_*$,
and the $\sigma^m-1-Q_s(\sigma)$ and $\sigma$ for a braided reflection around some
$H \in x_*$. Therefore $\kk \hat{B}_{x_*}/\mathfrak{J}_X = \kk \tilde{H}_x $
whence, from the isomorphism $\kk B \ltimes \kk \mathcal{L} \simeq
\bigoplus_{X \in \mathcal{L}/W} Mat_X(\kk \hat{B}_{x_*})$
we get the following.
\begin{theorem}  \label{theo:structtheorem} Let $\mathcal{L}$ be an admissible lattice. Then we have an isomorphism
$$
\mathcal{C}_{\kk}(W,\mathcal{L}) \simeq 
\bigoplus_{X \in \mathcal{L}/W} Mat_X(\kk \tilde{H}_{x_*}).
$$
\end{theorem} 

The following corollary completes the proof of theorem \ref{theo:structtheoremBMRintro}. 
\begin{corollary}
The algebra 
$\mathcal{C}_{\kk}(W,\mathcal{L}) $
 is a free $\kk$-module of
finite rank if and only if the BMR freeness conjecture holds over $\kk$ for every $x \in \mathcal{L}$.
In that case, its rank is $|W| \times |\mathcal{L}|$, and it is generically semisimple.
\end{corollary}

The fact that it is generically semisimple is a consequence of the fact
that, under the specialization morphism $\varphi : R \to \Q$ defined by $a_{H,i} \mapsto 0$
if $i > 0$, $a_{H,0} \mapsto 0$, the algebra $\mathcal{C}(W,\mathcal{L}) \otimes_{\varphi} \Q$
becomes isomorphic to a semidirect product $\Q W \ltimes \Q \mathcal{L}
\simeq \bigoplus_{\mathcal{L}/W} Mat_X(\Q W_{x_*})$, where $W_{x_*} < W$ is the
stabilizer of $x_* \in \mathcal{L}$. By Maschke's theorem we get that $\mathcal{C}(W,\mathcal{L}) \otimes_{\varphi} \Q$ is semisimple, and therefore $\mathcal{C}(W,\mathcal{L})$
is generically semisimple as soon as it is a free $R$-module of finite rank.
By Tits' deformation theorem we get that 
$$\mathcal{C}_{\bar{K}}(W,\mathcal{L}) \simeq \bar{K} W \ltimes \bar{K} \mathcal{L}
\simeq \bigoplus_{\mathcal{L}/W} Mat_X(\bar{K} W_{x_*}).
$$

Since the BMR freeness conjecture is now proved for all irreducible reflection groups
but the ones of Shephard-Todd types $G_{17}$, $G_{18}$ and $G_{19}$ (see
\cite{ARIKI,
ARIKIKOIKE,CYCLO,HECKECUBIQUE,MARINPFEIFFER,CHAVLITHESE,CHAVLICRAS,CB3}),
this proves the
following.

\begin{corollary}
The algebra 
$\mathcal{C}_{\kk}(W,\mathcal{L}) $
is a free $\kk$-module of
rank $|W| \times |\mathcal{L}|$, and is generically semisimple, except possibly if there exists $x \in \mathcal{L}$
whose associated reflection group admits an irreducible component of Shephard-Todd type
$G_{17}$, $G_{18}$ or $G_{19}$.
\end{corollary}

\section{Traces}
\label{sect:traces}

In this section, we slightly extend the ring of definition, for convenience. For
$W$ a given complex reflection group, we denote $R_u = \Z[u_{c,i}^{\pm 1}]$,
where $c$ runs among the conjugacy classes of distinguished pseudo-reflections,
and $i$ between $1$ and the order of (a representative of) $c$. We consider $R$
as a subring of $R_u$ where $a_{H,i}, H \in \mathcal{A}$ is mapped to the $(m_H-i)$-th symmetric
function in the $u_{c,k}$, where $c$ is the conjugacy class corresponding to
the distinguished pseudo-reflection with hyperplane $H$.
We let $H_u$ denote the Hecke algebra of $W$ defined over $R_u$,
that is $H_u = H(W) \otimes_R R_u$.

\subsection{Reminder on canonical traces}
Let $W$ be a complex reflection group, $B$ its braid group, $H = H_u$ its Hecke
algebra, defined over the ring of definition $R_u = \Z[u_{c,i}^{\pm 1}]$.
Let $x \mapsto \bar{x}$ the automorphism of $R_u$ defined by $u_{c,i} \mapsto u_{c,i}^{-1}$.
The group antiautomorphism $g \mapsto g^{-1}$ on $B$
induces an antiautomorphism of $\Z$-algebras $\mathbf{a} : R_u B \to R_u B$
such as $\mathbf{a}(\la g) = \bar{\la} g^{-1}$ for all $\la \in R_u$ and $g \in B$.
The Hecke ideal $\mathfrak{J}_H$ of $R_u B$ is stable by $\mathbf{a}$ 
hence $\mathbf{a}$ induces an automorphism of $H$. It has the property that,
for all parabolic subalgebras $H_0$ of $H$, $H_0$ is $\mathbf{a}$-stable
and the restriction of $\mathbf{a}$ to $H_0$ coincides with the antiautomorphism
associated to $H_0$. Let $t : H \to R$ be a linear form. We assume that $H$
admits a $R_u$-basis whose elements are (images of) elements of $B$. Note that this
is proved so far for all complex reflection groups but the ones having an irreducible component of type $G_{17}$, $G_{18}$ or $G_{19}$. We denote $\bpi$ the natural central
element of $P = \Ker( B \onto W)$. We consider the following assumptions on $t$.
\begin{enumerate}
\item $t$ is a symetrizing trace on $H$.
\item The trace induced on the specialization $\C W$ of $H$ is
the usual trace on the group algebra $\C W$
\item For all $h \in H$, we have $\overline{t(\mathbf{a}(h))} t(\bpi) = t(h \bpi)$.
\end{enumerate}

In \cite{SPETSES1} proposition 2.2 it is proven that, if there exists a trace
satisfying these assumptions, then it is unique. 
It is also proven there that, in case $W$ is a Coxeter group, then the trace given
by $t(T_w) = 0$ if $w \neq 1$, $t(T_1) = 1$, satisfies these assumptions.

\subsection{Traces on generalized Hecke algebras}

Let $\mathcal{L}$ be an admissible lattice, and $x \in \mathcal{L}$. Let $W_0$
denote the full reflection subgroup attached to $x$ and $H_0$ the corresponding
Hecke algebra. We already proved that the generalized Hecke algebra
$\tilde{H}_x$ attached to $x$ is a free $H_0$-module of the form $\bigoplus_{i=1}^m b_i H_0$
where the $b_i$ are (classes inside $\tilde{H}_x$ of) representatives of $B_x/B_0
\simeq G_x/W_0$. Obviously one can assume $b_1 = 1$ hence $b_1 H_0 = H_0$.
Assume that we are given a trace $t : H_0 \to R_u$ satisfying the conditions
of the previous section. We extend it as a linear form $t : \tilde{H}_x \to R_u$
by $t(b_i H_0) = 0$ if $i > 1$. 
\begin{proposition} The extended linear form $t  : \tilde{H}_x \to R_u$
is a symmetrizing trace.
\end{proposition}
\begin{proof}
In order for it to be a trace one needs to check
that for all $a_1,a_2 \in H_0$ and $i,j$ we have $t(b_i a_1 b_j a_2) = 
t(b_j a_2 b_i a_1)$. But clearly both terms are $0$ if $b_j \not\in b_i^{-1} H_0$.
Therefore we need to check that $t(b_i a_1 b_i^{-1} a_2) = 
t(b_i^{-1} a_2 b_i a_1)$ for all $i$ and $a_1,a_2 \in H_0$. But this means
$t(b_i a_1 (b_i^{-1} a_2b_i) b_i^{-1}) = 
t(b_i^{-1} a_2 b_i a_1)$. Since $a_2 \mapsto b_i^{-1} a_2 b_i$
induces a bijection of $H_0 \into \tilde{H}_x$ this is equivalent to
saying that $t(b_i a_1 a_2 b_i^{-1}) = t(a_2 a_1)$ for all $a_1,a_2 \in H_0$.
But $t(a_2 a_1) = t(a_1 a_2)$ whence we need to check that, for all $i$ and
all $a \in H_0$, we have $t(b_i a b_i^{-1}) = t(a)$. This holds true for the following
reason. Let $b \in B_x$, and consider the map $a \mapsto t(bab^{-1})$. This
is a trace on $H_0$, which satisfies obviously the conditions (1) and (2) of
the previous section. It also satisfies condition (3) if we can prove
that $b \bpi_0 b^{-1} = \bpi_0$ where $\bpi_0$ is the natural central element
of the pure braid group $P_0$ of $W_0$. But this was proven in section \ref{sect:braidsubgroups} above.
Therefore $t$ is a trace on $\tilde{H}_x$. Taking a basis $e_1,\dots,e_N$
of $H_0$ and letting $e'_1,\dots,e'_N$ its dual basis, so that $t(e_i e'_j) = \delta_{ij}$,
we get that the $b_i e_j$ form a basis for $\tilde{H}_x$,
with dual basis $e'_j b_i^{-1}$. Indeed, $t(b_i e_j e'_r b_s^{-1}) = 
t(b_s^{-1}b_i e_j e'_r ) = 0$ unless $i \neq s$, and in that case it is equal to
$t(e_j e'_r) = \delta_{jr}$. Therefore $t$ is a symmetrizing trace.
\end{proof}

\subsection{Symmetrizing trace}
We recall the following standard property of traces on matrix algebras, the proof being easy
and left to the reader.
\begin{lemma} Let $\kk$ be a commutative ring with $1$, $A$ a $\kk$-algebra
and $N \geq 1$. There is a 1-1 correspondence between trace forms on $A$
and trace forms on $Mat_N(A) = Mat_N(\kk) \otimes_{\kk}A$, the correspondence being given by $t \mapsto tr \otimes t$, where $tr : Mat_N(\kk) \to \kk$ is the matrix trace.
Moreover $tr \otimes t : Mat_N(\kk) \otimes_{\kk}A \to \kk \otimes_{\kk} \kk = \kk$
is symmetrizing if and only if $t$ is symmetrizing. 
\end{lemma}

From the isomorphism $(\kk B \ltimes \kk \mathcal{L})/\mathfrak{J} \simeq \bigoplus_{X \in \mathcal{L}/W} Mat_X(\kk \tilde{H}_{x_*})$ we are able to construct a trace form,
as 
$$
\bigoplus_{X \in \mathcal{L}/W} t_{x_*} \otimes tr : 
\bigoplus_{X \in \mathcal{L}/W} Mat_X(\kk \tilde{H}_{x_*})
= \bigoplus_{X \in \mathcal{L}/W} Mat_X(\kk )\otimes_{\kk} \kk \tilde{H}_{x_*} \to \kk
$$
and by the above property it is a symmetrizing form. This proves the following.

\begin{theorem} Let $\mathcal{L}$ be an admissible lattice for $W$, and
$\kk$ a commutative $R_u$-algebra. If the Broué-Malle-Michel
trace conjecture holds for all $x \in \mathcal{L}$,
then
the algebra 
$\mathcal{C}_{\kk}(W,\mathcal{L}) $
is a symmetric algebra. 
It is in particular the case when $W$ is a real reflection group.
\end{theorem}

\section{Main examples}

We recall that a reflection subgroup $W_0$ of $W$ is called \emph{full} if, for
every reflection $s \in W_0$, all the reflections with respect to the same reflecting hyperplane
belong to $W_0$. Such a reflection subgroup is uniquely determined by
the set of its reflecting hyperplanes. Of course reflection subgroups of real reflection groups and, more generally,
of 2-reflection groups, are full.

Let $\mathcal{L}_{\infty}$ denote the poset of all full reflection subgroups, ordered
by inclusion. For convenience, we prefer to consider it as a poset of subsets
of $\mathcal{L}$, also ordered by inclusion.

Recall that a subset $\mathcal{L} \subset \mathcal{L}_{\infty}$ is called \emph{admissible} if it is a sub-join-semilattice  of $\mathcal{L}_{\infty}$ 
which satisfies the following conditions: 
\begin{enumerate}
\item It is $W$-stable
\item It contains all $\{ H \}$, for $H \in \mathcal{A}$, as well as the trivial subgroup.
\end{enumerate}
Because such an $\mathcal{L}$ always contains a minimal element (the trivial group), there is no ambiguity in the definition of the semi-lattice : the fact that $a \vee b$ exists for every two elements of
$\mathcal{L}$ is in this case equivalent to saying that every finite subset of elements, including the
empty one, admits a join. Moreover, since such an $\mathcal{L}$ is always finite,
it is automatically a lattice. Therefore, we can equivalently talk about
admissible \emph{lattices}.

\subsection{The category of admissible semi-lattices and maps}

Let $\mathcal{L}$ and $\mathcal{L}'$ be two admissible semi-lattices.
A map $\mathcal{L} \to \mathcal{L}'$  is called admissible if it
is a $W$-equivariant morphism of join semi-lattices which is the identity
on the cyclic and trivial reflection subgroups. The collection of admissible
semi-lattices with morphisms the admissible maps forms a (small, finite) category $\mathcal{CL}_W$,
and $\mathcal{C}_{\kk}(W,\bullet)$ defines a functor from $\mathcal{CL}_W$ to
the category of (associative, unital) $\kk$-algebras. The category $\mathcal{CL}_W$
admits a terminal object that we call $\mathcal{L}_2$ : it is the subset of $\mathcal{L}_{\infty}$
made of the trivial and cyclic reflection subgroups together with the whole group $W$.
Obviously, for every admissible $\mathcal{L}$ there exists exactly one admissible map $\mathcal{L}
\to \mathcal{L}_2$. In particular there exists exactly one admissible map $\mathcal{L}_{\infty}
\to \mathcal{L}_2$.

More generally, define the parabolic rank of a reflection subgroup 
$W_0$ as the rank of the smallest parabolic subgroup containing $W_0$,
or equivalently as the codimension of its set of fixed points. Then,
the sub-poset $\mathcal{L}_n$ made of all reflection subgroups of parabolic rank at most $n$ plus the whole group is an admissible semi-lattice as soon as $n \geq 2$,
and there is an admissible map $\mathcal{L}_m \to \mathcal{L}_n$ when $m \geq n$
given by $W_0 \mapsto W_0$ if $W_0$ has parabolic rank
 at most $n$, and $W_0 \mapsto W$
if $W_0$ has rank at least $n+1$. This applies to $m = \infty$ as well.

\subsection{The semi-lattice $\mathcal{L}_2$}

The $W$-orbits of $\mathcal{L}_2$ are $\{ \{ 1 \} \}$, $\{ W \}$ together with
the $b_c = \{ \{ H \} ; H \in c \}$ for every $c \in \mathcal{A}/W$. It is immediately
checked that $\kk \tilde{H}_{1} = \kk W$ and $\kk \tilde{H}_{W_*} = H_{\kk}(W)$.
From theorem \ref{theo:structtheorem} we get that
$$
\mathcal{C}_{\kk}(W,\mathcal{L}_2) = \kk W \oplus 
H_{\kk}(W)
\oplus \bigoplus_{c \in \mathcal{A}/W} Mat_{|c|}(
\kk \tilde{H}_{c_*})
$$

A remarkable fact about the $x = \{ H \} \in \mathcal{L}$ of rank $1$, for any admissible poset,
is that the generalized Hecke algebras $\tilde{H}_{x}$ are free deformations of
the group algebra $\kk G(H)$, where $G(H) = \{ w \in W \ | \ w(H) = H \}$,
without having to invoque the BMR freeness conjecture (or, said differently,
it corresponds to the trivial case (rank 1) of the BMR freeness conjecture).

\subsection{The case of finite Coxeter groups}

Assume that $W$ is a real reflection group, and let $(W,S)$ be a
Coxeter system attached to it. Then $B$
admits a presentation as an Artin group, with generators $b_s, s \in S$.
The map $B \to W$ admits a natural set-theoretic section, called
Tits' section, and defined by $w \mapsto b_w = b_{s_1} \dots b_{s_n}$
where $s_i \in S$ and $w = s_1 \dots s_n$ is an expression of $w$ 
as a product of the generators of minimal length. The classical theory
tells us that it is well-defined. We denote $g_w$ the image of $b_w$
inside $\mathcal{C}(W,\mathcal{L})$ under the natural $R$-algebra morphism $R B \to
\mathcal{C}(W,\mathcal{L})$. 

Since the BMR freeness conjecture
is true for all reflection subgroups of $W$, from theorem \ref{theo:structtheoremBMRintro} we
know that $\mathcal{C}(W,\mathcal{L})$ is a free $R$-module of rank $|W| \times |\mathcal{L}|$. More precisely, we have the following.

\begin{proposition} Let $W$ be a finite Coxeter group and $\mathcal{L}$
an admissible lattice. Then $\mathcal{C}(W,\mathcal{L})$ admits for
basis the elements $g_w e_L$ for $w \in W$ and $L \in \mathcal{L}$.
\end{proposition}
\begin{proof}
Since the collection $\{ g_w e_L ; w \in W, L \in \mathcal{L}\}$ has the right
cardinality, it is sufficient to prove that it spans the free $R$-module
of finite rank $\mathcal{C}(W, \mathcal{L})$. For this we consider its
span that we denote $V$ ; we remark that $1 \in V$, and prove that
it is a left ideal of the $R$-algebra $\mathcal{C}(W, \mathcal{L})$.
Since the $g_s, s \in S$ and $e_L, L \in \mathcal{L}$ generate $R B \ltimes R \mathcal{L}$
as an algebra,
they also generate $\mathcal{C}(W, \mathcal{L})$ and therefore
it is sufficient to show that $g_s.x \in V$ and $e_L.x \in V$ for $x$ running among a spanning
set of $V$. Setting $x = g_w e_M$ for some $w \in W, M \in \mathcal{L}$, we get
$e_L g_w e_M = g_w e_{w^{-1}(L)} e_M =g_w e_{w^{-1}(L) \vee M} \in V$.
Let $\ell : W \to \N = \Z_{\geq 0}$ denote the classical length function. If $\ell(sw) = \ell(w)+1$,
then $g_s x = g_s g_w e_M = g_{sw} e_M \in V$. If not,
$w$ can be written $w = s w'$ with $\ell(w') = \ell(w)-1$. Then $g_s g_w = 
g_s^2 g_{w'} = g_{w'} + (u_s-1) e_{\langle s \rangle } (1+g_s) g_{w'} = 
g_{w'} + (u_s-1) e_{\langle s \rangle } g_{w'} + (u_s-1) e_{\langle s \rangle } g_s g_{w'}
= g_{w'} + (u_s-1) g_{w'} e_{\langle s^{w'} \rangle }  + (u_s-1) e_{\langle s \rangle } g_w$,
hence $g_s g_w e_M =  g_{w'}e_M + (u_s-1) g_{w'} e_{\langle s^{w'} \rangle }e_M  + (u_s-1) e_{\langle s \rangle } g_we_M
=  g_{w'}e_M + (u_s-1) g_{w'} e_{\langle s^{w'} \rangle \vee M}  + (u_s-1) e_{\langle s \rangle } g_we_M \in V$. This proves the claim.
\end{proof}

This proposition implies the following corollary, which could also be directly obtained
from the approach of \cite{YH1}  -- for instance by extending the left action of $\mathcal{C}_W$
on itself to an action of $\mathcal{C}(W,\mathcal{L}_{\infty})$.

\begin{corollary} If $W$ is a finite Coxeter group, then $\mathcal{C}_W \simeq \mathcal{C}(W,\mathcal{L}_{\infty})$.
\end{corollary}
\begin{proof} The elements $g_s, s \in S$ and $e_H, H \in \mathcal{A}$
clearly satisfy inside $\mathcal{C}(W,\mathcal{L}_{\infty})$ the defining relations of $\mathcal{C}_W$, and from this we get
an algebra morphism $\mathcal{C}_W \to \mathcal{C}(W, \mathcal{L}_{\infty})$. From
the above proposition and theorem 3.4 of \cite{YH1} we get that it maps a
basis of $\mathcal{C}_W$ to a basis of $\mathcal{C}(W,\mathcal{L}_{\infty})$,
and therefore it is an isomorphism.
\end{proof}

Therefore, the construction of $\mathcal{C}(W,\mathcal{L}_{\infty})$ indeed generalizes
to the complex reflection group case  the algebra $\mathcal{C}_W$ of a finite Coxeter group introduced in \cite{YH1}.

\subsection{The parabolic lattice}

A $W$-stable subposet of $\mathcal{L}_{\infty}$ is given by the collection $\mathcal{L}_p$
of parabolic subgroups. It can be identified with 
the arrangement lattice $L(\mathcal{A})$,
that is the collection of all intersections of hyperplanes in $\mathcal{A}$,
ordered by reverse inclusion. More precisely, there exists a map
$\mathrm{Fix} : \mathcal{L} \to L(\mathcal{A})$ where
$\mathrm{Fix}(x)$ is the intersection of all reflecting hyperplanes in $x$, and its
restriction to $\mathcal{L}_p$ is a bijection.

\begin{proposition} For $x \in \mathcal{L}_{\infty}$ a reflection subgroup,
let $[x] \in \mathcal{L}_p$ denote the parabolic closure of $x$. Then
$x \mapsto [x]$ is an admissible map $\mathcal{L}_{\infty} \to \mathcal{L}_p$
inducing a quotient map $\mathcal{C}(W,\mathcal{L}_{\infty}) \to \mathcal{C}(W,
\mathcal{L}_p)$.
\end{proposition}
\begin{proof}
First note that, for every $E,F \subset W$,
we have $\mathrm{Fix}(E \cup F) = \mathrm{Fix}(E) \cap \mathrm{Fix}(F)$,
$\mathrm{Fix}(E) = \mathrm{Fix}(\langle E \rangle)$, and
 $\mathrm{Fix}(x) = \mathrm{Fix}([x])$
if $x$ is a reflection subgroup. From this we get that,
for all $x,y \in \mathcal{L}$, we have
on the one hand $\mathrm{Fix}([ \langle x, y \rangle ]) = \mathrm{Fix}(\langle x \cup y \rangle)
= \mathrm{Fix}( x \cup y) = \mathrm{Fix}(x) \cap \mathrm{Fix}(y)$,
and on the other hand
$\mathrm{Fix}([ \langle [x] \cup [y] \rangle ])
\mathrm{Fix}(\langle [x] \cup [y] \rangle)
= \mathrm{Fix}([x] \cup [y]) = \mathrm{Fix}([x])\cap \mathrm{Fix}([y])
= \mathrm{Fix}(x) \cap \mathrm{Fix}(y)$.
Since $\mathrm{Fix}$ is a bijection $\mathcal{L}_p \to L(\mathcal{A})$
this proves $[ \langle x, y \rangle ] = [ \langle [x] \cup [y] \rangle ]$,
and this proves the claim, the $W$-invariance being obvious.

\end{proof}

From this we recover the definition of $\mathcal{C}_W^p = \mathcal{C}(W,\mathcal{L}_p)$
given in \cite{YH1} in the case of a finite Coxeter group, and extend 
the map $\mathcal{C}(W,\mathcal{L}_{\infty}) \to \mathcal{C}_W^p$
to the complex reflection group case.

\subsection{Root systems}
 Let $R$ be a reduced root system (in the sense of Bourbaki), $W$ the
 associated real reflection group. To each $\alpha \in R$ we associate the corresponding
 reflection $s_{\alpha} = s_{-\alpha} \in W$. A root subsystem of $R$ is by
 definition a subset $R'$ of $R$ stable under every $s_{\alpha}, \alpha \in R'$. The
 subgroup of $W$ generated by the $s_{\alpha}$ for $\alpha \in R'$ is a reflection
 subgroup, and the map $R' \mapsto \langle s_{\alpha}, \alpha \in R'$ defines
 a bijection between the set $\mathcal{L}_R$ of all root subsystems and
 $\mathcal{L}_{\infty}$. The preordering induced by this bijection on $\mathcal{L}_R$
 is simply the inclusion ordering. We endow $\mathcal{L}_R$ we the corresponding
join semilattice structure. The cyclic reflection subgroups of $W$ correspond
to the root subsystems $\{ \alpha, - \alpha \}$ for $\alpha \in R$.
 
 We let $\mathcal{L}_c$ denote the subset of $\mathcal{L}_R$ corresponding
 to the \emph{closed} subsystems, namely the $R' \in \mathcal{L}_{\infty}$ for which $\forall \alpha,
 \beta \in R'  \ \ \ \alpha + \beta \in R \Rightarrow \alpha + \beta  \in R'$. Note
 that an intersection of closed subsystems is a closed subsystem, and that the
 subsystems of the form $\{ \alpha, - \alpha \}$ as well as the
 empty subsystem are closd.
 We have a map $c : \mathcal{L}_R \to \mathcal{L}_c$ which associates to
 $R' \in \mathcal{L}_R$ its \emph{closure}, namely the intersection of all closed 
 subsystems containing it. It is immediately checked that $c$ is $W$-equivariant
  and a join semilattice morphism. From this it follows that we get
  an admissible map $\mathcal{L}_{\infty} \simeq \mathcal{L}_R \to \mathcal{L}_c$.

This proves the following.

\begin{proposition} Let $R$ be a reduced root system and $W$ the associated finite Coxeter group. Under
the identification $\mathcal{L}_{\infty} \simeq \mathcal{L}_R$,
the map $c : \mathcal{L}_R \to \mathcal{L}_c$ induces a surjective morphism
$\mathcal{C}(W,\mathcal{L}_{\infty}) \to \mathcal{C}(W,\mathcal{L}_c)$.
\end{proposition}

This proposition proves that the algebra $\mathcal{C}(W,\mathcal{L}_c)$
is isomorphic to the algebra $\mathcal{C}_W^R$ of \cite{YH1},
which generically embeds into the corresponding
Yokonuma-Hecke algebra. Indeed, 
$\mathcal{C}_W^R$ is defined as a quotient of $\mathcal{C}(W,\mathcal{L}_{\infty}) =
\mathcal{C}_W$, and one gets immediately that the map $\mathcal{C}_W \onto
\mathcal{C}(W,\mathcal{L}_c)$ defined above factors through $\mathcal{C}_W \onto
\mathcal{C}_W^{(R)}$. The induced surjective
map $\mathcal{C}_W^{(R)} \to \mathcal{C}(W,\mathcal{L}_c)$ is then checked
to be injective, since the natural spanning set of $\mathcal{C}_W^{(R)}$
is mapped to a basis of $\mathcal{C}(W,\mathcal{L}_c)$. It is then immediately
checked that the corresponding diagram of isomorphisms and natural projections
is commutative.

$$
\xymatrix{
\mathcal{C}_W \ar[rr]^{\simeq} \ar@{->>}[dr] \ar@{->>}[dd] & & \mathcal{C}(W,\mathcal{L}_{\infty})\ar@{->>}[dd] | \hole \ar@{->>}[dr] \\
 & \mathcal{C}_W^{(R)} \ar@{->>}[dl] \ar[rr]^{\simeq}  & & \mathcal{C}(W,\mathcal{L}_c) \ar@{->>}[dl] \\
 \mathcal{C}_W^{(p)} \ar[rr]^{\simeq} & & \mathcal{C}(W, \mathcal{L}_p)
}
$$

\subsection{A priori unrelated examples}
A computer-aided exploration shows that there are other
admissible lattices not originating a priori from root systems, with $\mathcal{L}_p 
\subset \mathcal{L} \subset \mathcal{L}_{\infty}$. In type $A$
we have $\mathcal{L}_p = \mathcal{L}_{\infty}$, but in type $D_n$ for $n \geq 4$
we have $\mathcal{L}_p \subsetneq \mathcal{L}_{\infty}$ while all root subsystems
are closed. We checked for small $n$ whether there are other admissible lattices in type $D_n$. This can be done as follows. First of all, one computes the $W$-orbits
for the action on $\mathcal{L}_{\infty} \setminus \mathcal{L}_p$,
since $\mathcal{L}_{\infty} \setminus \mathcal{L}$ has to be an union of them.
For each such union of orbits we then test whether the obtained subset $\mathcal{L}$
satisfies the join semilattice property. In type $D_4$, the action of $W$
on $\mathcal{L}_{\infty} \setminus \mathcal{L}_p$ is transitive (and it the
orbit of a reflection subgroup of type $A_1^4$), so there
is no intermediate admissible lattice. But in type $D_5$, the action has 2 orbits,
one of type $A_1^4$ inherited from type $D_4$, and the other one of type $A_1A_3$.
By adding to $\mathcal{L}_p$ the orbit of type $D_4$ one checks by
computer that the corresponding poset $\mathcal{L}$ is admissible, every
two elements admitting a join. This proves
that examples containing the lattice of parabolic subgroups and which are a priori not related to the theory of root systems do exist.


\begin{thebibliography}{00}
\bibitem{AICARDIJUYU} F. Aicardi, J. Juyumaya, {\it An algebra involving braids and ties}, ICTP Preprint IC/2000/179.
\bibitem{AICARDIJUYU2} F. Aicardi, J. Juyumaya, {\it Markov trace
on the algebra of braids and ties}, Moscow Math. J. 16(3)
(2016) 397–431.
\bibitem{ARIKI} S. Ariki, {\it Representation theory of a Hecke algebra of $G(r,p,n)$}, J. Algebra {\bf 177} (1995), 164--185.
\bibitem{ARIKIKOIKE} S. Ariki, K. Koike, {\it A Hecke algebra of $(\Z/r\Z)\wr \mathfrak{S}_n$
and construction of its irreducible representations}, Advances in Math. {\bf 106} (1994), 216–243.
\bibitem{BESSIS} D. Bessis, {\it Variations on Van Kampen's method}, J. Math. Sci. (N.Y.) {\bf 128} (2005), 3142--3150.
\bibitem{CHAVLITHESE} E. Chavli, {\it The BMR freeness conjecture for
exceptional groups of rank 2}, doctoral thesis, Univ. Paris Diderot (Paris 7), 2016.
\bibitem{CHAVLICRAS} E. Chavli, 
{\it The BMR freeness conjecture for the first two families of the exceptional groups of rank 2},
to appear in Comptes Rendus Math\'ematiques.
\bibitem{CB3} E. Chavli, {\it Universal deformations of the finite quotients of the braid group on 3 strands}, J. Algebra {\bf 459} (2016), 238-271.
\bibitem{SPETSES1} M. Brou\'e, G. Malle, J. Michel, {\it Towards Spetses I},
Transform. Groups {\bf 4} (1999), 157–218.
\bibitem{BMR} M. Brou\'e, G. Malle, R. Rouquier,
{\it Complex reflection groups, braid groups, Hecke
algebras}, J. Reine Angew. Math. {\bf 500} (1998), 127--190.
\bibitem{CABANESENGUEHARD} M. Cabanes, M. Enguehard,  {\it Representation theory of finite reductive groups}, Cambridge, 2004.
\bibitem{CABANESMARIN} M. Cabanes, I. Marin, {\it On ternary quotients of cubic Hecke algebras},
Comm. Math. Phys. {\bf 314} (2012), 57–92. 
\bibitem{COGOLLUDO} J.I. Cogolludo-Agust\'\i n, {\it Braid monodromy of algebraic curves},
 Ann. Math. Blaise Pascal {\bf 18} (2011), 141--209. 
\bibitem{ERH} J. Espinoza, S. Ryom-Hansen, {\it Cell structures for the Yokonuma-Hecke algebra and the algebra of braids and ties}, preprint 2016, arxiv 1503.03396v3.
\bibitem{HECKECUBIQUE} I. Marin, {\it The cubic Hecke algebra on at most 5 strands},  J. Pure Appl. Algebra {\bf 216} (2012), 2754–2782.
\bibitem{KRAMCRG} I. Marin, {\it Krammer representations for complex braid groups},
J. Algebra {\bf 371} (2012), 175--206.
\bibitem{CYCLO} I. Marin, {\it The freeness conjecture for Hecke algebras of complex reflection groups, and the case of the Hessian group $G_{26}$}, J. Pure Appl. Algebra {\bf 218} (2014),  704–720.
\bibitem{YH1} I. Marin, {\it Artin groups and the Yokonuma-Hecke algebras},
preprint 2016, to appear in I.M.R.N.
\bibitem{G20G21} I. Marin, {\it Proof of the BMR freeness conjecture for the groups $G_{20}$
and $G_{21}$}, preprint 2017, arxiv:1701.09017.
\bibitem{MARINPFEIFFER} I. Marin, G. Pfeiffer, {\it The BMR freeness conjecture for the 2-reflection groups}, preprint 2014, to appear in Math. of Comput.
\bibitem{OT} P. Orlik, H. Terao, {\it Arrangements of hyperplanes},  Springer-Verlag, Berlin, 1992.
\bibitem{RYOMHANSEN} S. Ryom-Hansen, {\it On the representation theory of an algebra of braids and ties}, J. Algebraic Combin. {\bf 33} (2011), no. 1, 57–79. 
\bibitem{SERRE} J.-P. Serre, {\it Linear representations of finite groups}, GTM {\bf 42}, Springer, 1977.
\bibitem{STANLEY} R. P. Stanley, {\it Enumerative combinatorics, Volume 1}, 2nd edition, Cambridge University Press, 2012.
\bibitem{YOKO} T. Yokonuma, {\it Sur la structure des anneaux de Hecke d'un groupe
de Chevalley fini}, C. R. Acad. Sci. Paris S\'erie A-B {\bf 264} (1967), A344-A347.
\end{thebibliography}
\end{document}